\documentclass[12pt]{amsart} 
\usepackage{amsmath,amssymb,amsfonts,amsthm}
\usepackage{verbatim}
\usepackage[dvips]{graphicx}

\usepackage{latexsym} 
\usepackage{amssymb,amsfonts,amsthm,mathtools}
\usepackage{verbatim}
 
\usepackage{textcomp}
\usepackage[all]{xy}
\usepackage{times}
\usepackage{graphicx,color}

\usepackage{latexsym} 
\usepackage{amsmath,amssymb,amsfonts,amsthm}
\usepackage{epic}
 \usepackage{tikz} 
 
\newtheorem{remark}{Remark}
\newtheorem{theorem}{Theorem}

\newtheorem{lemma}{Lemma}
\newtheorem{example}{Example}
\usepackage{pdfpages}

 \title{Weak solutions for stochastic differential equations with additive fractional noise}

\thanks{Research partially supported by FAPESP 2018/13481-0 and 2020/04426-6 and Coordenação de Aperfeiçoamento de
Pessoal de Nível Superior - Brasil (CAPES) - Finance Code 001.}


\author{Pedro J. Catuogno and Diego S. Ledesma }
\address{Departamento de
 Matem\'{a}tica, Universidade Estadual de Campinas,\\ 13.081-970 -
 Campinas - SP, Brazil.}
\email{pedrojc@unicamp.br and ledesma@unicamp.br  }

\begin{document}

\begin{abstract}
We give a new approach to prove the existence of a weak solution of 
\[
dx_t = f(t,x_t)dt + g(t)dB^H_t
\]
where $B^H_t$ is a fractional Brownian motion with values in a separable Hilbert space for suitable functions $f$ and $g$. Our idea is to use the implicit function theorem and the scaling property of the fractional Brownian motion in order to obtain a weak solution for this equation.
\end{abstract}

\maketitle

\section{Introduction}

Recently there has been a great interest in the study of stochastic partial differential equations (SPDE) driven by coloured noise. The main motivation comes from finance and physics. In finance, it is well known that the market evolves according to a geometric fractional $H-$Brownian motion with (see \cite{bender} \cite{grecksch} and \cite{elliot}). In physics, fractional mechanics and its generalization to stochastic mechanics have been developed in recent years (see \cite{mandelbrot} and  \cite{Meer}).

We are interested to deal with the following kind of SPDE
\begin{equation}
dx_t=f(t,x_t)~dt+g(t)~dB_t^H,\label{EDEs}
\end{equation} 
where $B_t^H$ is a fractional Brownian motion (fBm) with Hurst parameter $H\in(0,1)$, $f:[0,T]\times \mathbb{H}\rightarrow \mathbb{H}$ and  $g:[0,T]\rightarrow L(\mathbb{H},\mathbb{H})$ are continuous maps and a Hilbert space $\mathbb{H}$ which will be defined precisely in the section section 3 below. 

From a mathematical point of view, the problem is to guarantee the existence of solutions for the equation (\ref{EDEs}). There are, in the literature, different notions of solution for SPDE of this kind (for example \cite{daprato}, \cite{gawarecki} and \cite{pardoux}). In this article we show the existence of a weak solution for the equation (\ref{EDEs}). There are several works in this direction. Here, we highlight the contributions of D. Nualart and Y. Oukine in \cite{Nualart} that was improved by B. Boufoussi and Y. Ouknine in \cite{Boufoussi},  J. Snup\'arkov\'a in \cite{Snuparkova} and, more recently by Z. Li, W. Zhan and L. Xu in \cite{LiZhanXu}. In the above works the proofs  are based on the Girsanov Theorem and the  Picard iteration argument. Here we present an alternative idea to prove the existence of a weak solution using the scaling property of the Fractional Brownian motion. We adapt a result about deterministic dynamical systems of J. Robbin \cite{Robbin} to prove the existence of a weak solution to of equation (\ref{EDEs}) under certain hypotheses.

The work is organized as follow: in section 2 we give the basic definitions about fractional Brownian motion on Hilbert spaces and its integration theory following the references \cite{oksendal}, \cite{duncan}, \cite{grecksch}, \cite{mishura} and  \cite{nourdin}. Finally, in section 3 we stablish our main result and give a simple example.

\section{Fractional Brownian motion}

We start reviewing  the basic concepts on fractional Brownian motion (fBm) and stochastic integral with respect to fBm. 

Fix a standard filtered probability space $(\Omega,\mathcal{F},~\{\mathcal{F}_t,~t\in[0,T]\}, \mathbb{P})$, see for instance  \cite{oksendal}, \cite{mishura} and \cite{nourdin}.

A fractional brownian motion with Hurst parameter $H\in (0,1)$ is a continuous and centered Gaussian process $(B_t^H)_{t\geq 0}$ with covariance function 
\[
\mathbb{E}\left[B_t^HB_s^H\right]=\frac{1}{2}(s^{2H}+t^{2H}-|t-s|^{2H}).
\]
A very important result on fBm theory is the following integral representation
\[
B^H_t=\int_0^tK_H(t,s)~dB_s,
\]
where 
\begin{itemize}
\item $B_s$ is the standard Brownian motion.
\item 
\[
K_H(t,s):=\left\{
\begin{array}{lcc}
b_H\left[\left(\frac{t(t-s)}{s}\right)^{H-1/2}-(H-1/2)s^{1/2-H}\int_s^t(u-s)^{H-1/2}u^{H-3/2}~du.
\right]~&&H<1/2\\
&&\\
c_Hs^{1/2-H}\int_s^t|u-s|^{H-3/2}u^{H-1/2}~du&&H>1/2\\
\end{array}
\right.
\]
\item  
\[
b_H=\sqrt{\frac{2H}
{(1-2H)\beta(1-2H,H+1/2)}},\] 
\item 
\[
c_H=\sqrt{\frac{H(2H-1)}{\beta(2-2H,H-1/2)}}.
\]
and
\item
\[
\beta(a,b)=\frac{\Gamma(a+b)}{\Gamma(a)\Gamma(b)} 
\]
where $\Gamma(a)= \int_0^{\infty}x^{a-1}e^{-x}dx$ is the Gamma function. (see \cite{oksendal}, \cite{mishura} and \cite{nourdin}).
\end{itemize}

To construct the  fractional Brownian motion we consider the Schwartz space $\mathcal{S}(\mathbb{R})$ of rapidly decreasing smooth functions on $\mathbb{R}$ provided with the inner product $<\cdot,\cdot>_H$  given by

\[
<f,g>_H:=\left\{
\begin{array}{lcc}
H(2H-1)\int_0^T\int_0^Tf(s)g(t)|t-s|^{2H-2}~dsdt &\textrm{for} & H>1/2\\
&&
\\
 \int_0^T K_H^*f(s)K_H^*g(s)~ds&\textrm{for}&H<1/2.
\end{array}
\right.\]
where
\[
K_H^*f(s):=K(T,s)f(s)+\int_s^T(f(t)-f(s))\partial_tK_H(t,s)~dt.
\] 
Let $ \mathcal{S}'(\mathbb{R})$ be the dual of $\mathcal{S}(\mathbb{R})$ the space of tempered distributions on $\mathbb{R}$. We consider the mapping $f\in\mathcal{S}(\mathbb{R})\rightarrow \exp(1/2||f||_H^2)\in \mathbb{R}$. This map is positive definite in $\mathbb{S}(\mathbb{R})$ thus, by the Bochner-Minlos theorem (see for instance \cite[pp. 257]{O} ), we guarantee the existence of a probability measure $\mathbb{P}^H$ on the Borel $\sigma-$algebra $\mathcal{B}(\mathcal{S}'(\mathbb{R}))$ such that
\[
\int_{ \mathcal{S}'(\mathbb{R})} e^{i<\omega,f>}~d\mathbb{P}^H(\omega)=e^{-1/2||f||_H^2}\quad \quad \textrm{for all}~f\in\mathcal{S}(\mathbb{R}).
\]  
We will denote by $I_{[0,t]}$ the usual characteristic function of the interval $[0,t]$. We define $B^H:[0,T] \times  \mathcal{S}'(\mathbb{R}) \rightarrow \mathbb{R}$ by  
\[
B^H(t, \omega)=B_t^H(\omega)=<\omega, I_{[0,t]} >
\]
as an element of $L^2(\mathbb{P}^H)$.

It is a simple consequence of the Kolmogorov continuity theorem that $B^H$ is a fractional Brownian motion with Hurst parameter $H$.

In order to construct the Wiener integral with respect to the fBm, we follow \cite{oksendal}. Consider the functions  $R_{H}:[0,T]\times[0,T]\rightarrow [0,\infty)$ given by
\[
R_H(t,s)=\frac{1}{2}(s^{2H}+t^{2H}-|t-s|^{2H}),
\]
and the closure $\mathcal{H}$ of the vector space spanned by the functions 
\[\{R_H(t,~),~t\in[0,T]\}\] with respect to the scalar product
\[
<R_H(t,\cdot),R_H(s,\cdot)>:=R_H(t,s) ~ \mbox{ for all }t,s\in[0,T].
\]

The Wiener integral with respect to the fBm is defined as the linear extension  of the isometric map 
\[
R_H(t,~)\rightarrow B^H_t\quad \textrm{for all}~t \in[0,T],
\]
from $\mathcal{H}$ in $L^2(\mathbb{P}^H)$.

To define the stochastic integral of a function with respect to the fBm we  follow the usual path \cite{oksendal}, \cite{mishura} and  \cite{nourdin}. We start defining it for a simple function $\phi=\sum_i{a_i}\chi_{[t_i,t_{i+1})}$ as 
\[
\int_0^T\phi(s)~dB_s^H=\sum_i{a_i}(B_{t_{i+1}}^H-B_{t_i}^H).
\]
Then, we extend it to $L^2(\mathbb{P}^H)$ and we have the isometry
\begin{equation}
\mathbb{E}\left[\left<\int_0^Tf(s)~dB_s^H,\int_0^Tg(s)~dB_s^H\right>\right]=<f,g>_H,\quad \label{isometry}
\end{equation}
see \cite{oksendal} and \cite{gripenberg}.

Now, we review the construction of a fractional Brownian motion in a Hilbert space $\mathbb{H}$. 

Let $\mathbb{H}$ be a separable Hilbert space. We will  denote by $L(\mathbb{H})$ the space of continuous linear maps from $\mathbb{H}$ to $\mathbb{H}$. Let $Q\in L(\mathbb{H})$ be a symmetric, non negative trace class operator. 

Let $\{\lambda_n\}_{n\in\mathbb{N}}$ be a discrete family of eigenvalues of $Q$, counted with multiplicity, and let $\{e_n\}_{n\in\mathbb{N}}$ be the normalized eigenvectors, this is 
\[
Q(e_n)=\lambda_ne_n.
\]

Let $\{B_t^n\}$ be independent 1-dimensional fBm with parameter $H$. Then the series
\[
B_t^H=\sum_{n=1}^\infty \sqrt{\lambda_n}B_t^ne_n,\quad t\geq 0
\]
converges almost surely and in $L^p$ for $p\geq 1$ (see \cite{oksendal}, \cite{grecksch}, \cite{mishura} and \cite{nourdin} ). We say that $B_t^H$ is a  trace-class fractional Brownian motion in $\mathbb{H}$ with covariance $Q$ and Hurst parameter $H$.

 Let $g:[0,T]\rightarrow L(\mathbb{H})$ be  a continuous map. We define the integral of $g$ with respect to the trace-class fBm $B^H$ as 
\[
\int_0^tg(t)~dB_t^H:=\sum_{n,m=1}^\infty\left(\int_0^t\sqrt{\lambda_m}<g(t)e_m,v_n>~dB_t^m\right)v_n,
\]
where $\{v_m\}_{m\in\mathbb{N}}$ is any orthonormal basis  of $\mathbb{H}$.

Using the isometry (\ref{isometry}) we get that 
\begin{eqnarray*}
\mathbb{E}\left|\int_0^tg(t)~dB_t^H\right|^2&=&\sum_{n,m=1}^\infty\mathbb{E}\left|\int_0^t\sqrt{\lambda_m}<g(t)e_m,v_n>~dB_t^m\right|^2\\
&=&\sum_{n,m=1}^\infty   \lambda_m  <<g(t)e_m,v_n>,<g(t)e_m,v_n>>_H.
\end{eqnarray*}
Thus 
\begin{equation}
\mathbb{E}\left|\int_0^tg(t)~dB_t^H\right|^2=\left\{\begin{array}{lcc}
\sum_{j,k=1}^\infty   \lambda_k H(2H-1)\int_0^T\int_0^T<g(t)e_k, g(s)e_k> \times&&\\
 \times|t-s|^{2H-2}~dsdt&\textrm{if}&H\geq 1/2,\\
\\~\\
~\\
~\\
\sum_{j,k=1}^\infty   \lambda_k \int_0^T<K_H^*g(s)e_k,K_H^*g(s)e_k>~ds&\textrm{if}&H< 1/2,
\end{array}\right.
\end{equation}\label{estimativas}
where
\begin{eqnarray*}
<K_H^*g(s)e_k,K_H^*g(s)e_k>&=&K(T,s)^2||g(s)e_k||^2+\\
&&+2K(T,s)\int_s^T<g(s)e_k,(g(t)-g(s))e_k>\partial_tK_H(t,s)~dt+\\
&&+\int_s^T\int_s^T<(g(r)-g(s))e_k,(g(t)-g(s))e_k>\times \\
&&\times \partial_tK_H(t,s)\partial_tK_H(r,s)~~drdt.
\end{eqnarray*}

The next lemma is a consequence of the above set up and it shows  estimates that we will use below.
\begin{lemma}\label{lemma1} Let  $g ,~h :[0,T]\rightarrow L(\mathbb{H},\mathbb{H})$ be continuous functions such that 
\begin{equation}
 \sup_{r\in[0,T]}||g(r)e_j||^2_H<\infty\quad \textrm{and}\quad  \sup_{r\in[0,T]}||rh(r)e_j||^2_H<\infty. \label{hipotesis}
\end{equation}
Then for any positive real numbers $a$ and $k$, 
\begin{eqnarray*}
\sum_{j}\lambda_j\mathbb{E}\left| a^{Hk}\int_0^Tg(a^ks)(e_j)~d B_t^H \right|^2
&<&\infty\\
\sum_{j}\lambda_j\mathbb{E}\left| a^{Hk}\int_0^Tsh(a^ks) ~d B_t^H \right|^2
&<&\infty.\\
\end{eqnarray*}
\end{lemma}

\begin{proof}
It follows directly from the construction of the fractional Brownian motion $B^H$ and   (\ref{estimativas}) we get
\begin{eqnarray*}
\sum_{j}\lambda_ja^{2Hk}\mathbb{E}\left| \int_0^Tg(a^ks)(e_j)~d(B_t^H)^j\right|^2
&\leq &C(H,T)a^{2Hk}\sup_{r\in[0,T]}||g(a^kr)e_j||^2_H\sum_j\lambda_j\\ 
\sum_{j}\lambda_ja^{2Hk}\mathbb{E}\left| \int_0^Tsh(a^Ks)(e_j)~d(B_t^H)^j\right|^2
&\leq&C(H,T)a^{2Hk}\sup_{r\in[0,T]}||rh(a^kr)e_j||^2_H\sum_j\lambda_j \\
\end{eqnarray*} 
where for $H>1/2$,
\[
C(H,T)=H(2H-1)\int_0^T\int_0^T |t-s|^{2H-2}~dsdt
\]
 and for $H<1/2$ 
\begin{eqnarray*}
C(H,T)&=& \int_0^TK(T,s)^2~ds+2\int_0^T \int_s^T|K(T,s) \partial_tK_H(t,s)|~dt~ds\\
&&+ \int_0^T\int_s^T\int_s^T   |\partial_tK_H(t,s)\partial_tK_H(r,s)|~drdtds.
\end{eqnarray*}
Now, the hypothesis (\ref{hipotesis}) guarantee the result. 
\end{proof}

\section{Main result}

We consider stochastic diferential equations  over $\mathbb{H}$ of the following type
\begin{equation}
dx_t=f(t,x_t)~dt+g(t)~dB_t^H,\label{EDE}
\end{equation} 
where
\begin{itemize}
\item $B_t^H$ is a trace-class fractional Brownian motion in $\mathbb{H}$ with covariance $Q$ and Hurst parameter $H$.
\item $f:[0,T]\times \mathbb{H}\rightarrow \mathbb{H}$ and  $g:[0,T]\rightarrow L(\mathbb{H},\mathbb{H})$ are continuous maps.
\end{itemize}
We say that a process $x_t$ is a weak solution to equation (\ref{EDE}) starting at $x_0\in \mathbb{H}$ if
\[
x_t-x_0=\int_0^tf(s,x_s)~ds+\int_0^tg(s)~d\tilde B_s^H,
\]
for some fractional Brownian motion $\tilde B_s^H$ with Hurst parameter $H$.

\begin{theorem}\label{teorema} Let $\mathbb{H}$ and $\tilde{\mathbb{H}}$ be two Hilbert spaces such that $\mathbb{H}\subset \tilde{\mathbb{H}}$ densely an there is a family of linear operators $\{S_n\}_{n\in \mathbb{N}}$ such that 
\begin{itemize}
\item $S_n:\tilde{\mathbb{H}}\rightarrow \tilde{\mathbb{H}}$,
\item $S_n(x)\in \mathbb{H}$ for all $x\in \tilde{\mathbb{H}}$,
\item $||S_n-I_{\tilde{\mathbb{H}}}||\rightarrow 0$.
\end{itemize}
Let $f:[0,T]\times \mathbb{H}\rightarrow \tilde{\mathbb{H}}$  and $g:[0,T]\rightarrow L(\tilde{\mathbb{H}},\tilde{\mathbb{H}})$ be continuously differentiable  maps such that 
\begin{itemize}
\item $ |f(t,x)|+|\partial_tf(t,x)|\leq \phi(t)(1+|x|)$,
\item $|\partial_xf(t,x)v|^2\leq \phi(t)^2(1+|x|^2+|v|^2)$,
\item $|g(t)(v)|+|\partial_tg(t)(v)|\leq \phi(t)|v|$,
\end{itemize} for a continuous square integrable function $\phi:[0,T]\rightarrow \mathbb{R}$.

Then there exists a fractional Brownian motion  $\tilde B^H$  in $\tilde{\mathbb{H}}$ with covariance $Q$ and Hurst parameter $H$, which is a rescaling of the original $B^H$, and $X:[0,T]\times \Omega\times \mathbb{H}\rightarrow \tilde{\mathbb{H}}$ satisfying the equation
\[
X(t)-x_0=\int_0^tf(s,X(s))~ds+\int_0^tg(s)~d\tilde B^H_s,
\] for any $t\in [0,T]$.

Furthermore, the dependence of $X$ in $x_0$ is smooth.

\end{theorem}
\begin{proof} We fix $0<a<1$ and $k\in 2\mathbb{N}$  such that $kH>1.$  
For any $X\in  C([0,T],L^2(\Omega, \tilde{\mathbb{H}}))$ we define 
\[
F_n(X,a,x_0)(t)=S_nX(t)-\int_0^ta^kf(a^ks,x_0+S_nX(s))~ds-\int_0^ta^{Hk}g(a^ks)~dB_s^H.
\]
By Lemma \ref{lemma1} and Fubini's theorem
\begin{eqnarray*}
\mathbb{E}[||F_n(X,a,x_0)(t)||^2]&\leq& 8\mathbb{E}[||S_nX(t)||^2]\\
&&+8\int_0^tt(a^{2k}|\mathbb{E}[|| f(a^ks,x_0+S_nX(s))||^2]~dt\\
&&+8C(H,T)\int_0^t\sum_j\lambda_ja^{2Hk}||g(a^ks)e_j||^2_H~ds.
\end{eqnarray*}
Thus,
\[F_n:C([0,T],L^2(\Omega, \tilde{\mathbb{H}}))\times \mathbb{R}\times \mathbb{H}\rightarrow  C([0,T],L^2(\Omega, \tilde{\mathbb{H}})),\]
is well defined and $F_n(0,0,x_0)=0$.  

Now we claim that $F_n$ belong to $C^1$ and $\partial_XF_n(0,0,x_0)$ is invertible. In fact, we have that 
\begin{eqnarray*}
D(F_n)_{(X,a,x_0)}(W,b,y)(t)&=&S_nW(t)-\int_0^ta^kDf_{(a^ks,x_0+S_nX(s))}(S_nW(s),b,y)~ds\\
&&-\int_0^tka^{Hk+k-1}sb\partial_tg(a^ks)~dB_s^H,
\end{eqnarray*}
where $W\in C([0,T],L^2(\Omega, \tilde{\mathbb{H}}))$, $b\in\mathbb{R}$ and $y\in\mathbb{H}$. 

We observe that 
\[
Df_{(a^ks,x_0+S_nX)}(S_nW,b,y)=\partial_tf(a^ks,x_0+S_nX(s))ka^{k-1}b+\partial_xf(a^ks,x_0+S_nX(s))(y+S_nW),
\]
is an integrable function for each $n\in\mathbb{N}$ big enough. In fact,
\begin{eqnarray*}
\int_0^t|\partial_tf(a^ks,x_0+S_nX(s))|^2~dt&\leq&\int_0^t\phi(s)^2|S_nX(s)|^2~ds\\
&\leq&\int_0^t\phi(s)^2| X(s)|^2~ds+\int_0^t\phi(s)^2|S_nX(s)-X(s)|^2~ds,
\end{eqnarray*}
and
\begin{eqnarray*}
\int_0^t|\partial_xf(a^ks,x_0+S_nX(s))(y+S_nW)|^2~ds&\leq&\int_0^t\phi(s)^2(1+|x_0+S_nX(s)|^2)~ds+\\
&&+\int_0^t\phi(s)^2|y+S_nW|^2~ds.
\end{eqnarray*}
Thus $F_n$ belong to $C^1$.

We have that $F_n$ verifies the hypothesis of the the implicit function theorem (see \cite[pp. 121]{marsden} ) because
\[
\partial_XF_n(0,0,x_0)W(t)=S_nW(t).
\]
and  $S_n$ is invertible for $n$ big enough. This follows from 
\[
S_n=I+(S_n-I),
\]
with $||S_n-I||\rightarrow 0$.

By the implicit function theorem,  there is a neighborhood $U$ of $(0,x_0)\in [0,T]\times \mathbb{H}$ and a differentiable map $G_n:U\rightarrow  C([0,T],L^2(\Omega, \tilde{\mathbb{H}}))$ such that 
\[
F_n(G_n(a,x),a,x)=0\quad \textrm{for all} ~(a,x)\in U.
\]We set 
\[
X(t):=S_nG_n(x_0,\epsilon)(t/\epsilon^k).
\]
Then, considering the  self-similarity property $\tilde B_{a^kt}^H:=a^{kH} B_t^H$ (see \cite{mishura}, \cite[Prop. 2.2]{nourdin}) we get 
\begin{eqnarray*}
X(t)-\int_0^tf(r,x_0+X(r))~dr-\int_0^tg(r )~d\tilde B_r^H&=&S_nG_n(x_0,\epsilon)(t/\epsilon^k)\\
&&-\int_0^{t/\epsilon^k}\epsilon^k f(\epsilon^k s,x_0+S_nG_n(x_0,\epsilon)(s))~ds\\
&&-\int_0^{t/\epsilon^k}\epsilon^{kH} g(\epsilon^k s )~dB_{s}^H\\
&=&F_n(G_n(x_0,\epsilon),\epsilon,x_0)(t/\epsilon^k)=0.
\end{eqnarray*}
Thus
\[
\tilde X=x_0+X,
\]
is a solution of the equation (\ref{EDE}) that depends smoothly in $x_0$. 
 
\end{proof}

\begin{remark} To show the result above we just need the estimates 
\[
\int_0^tt(a^{2k}|\mathbb{E}[ f(a^ks,x_0+X(s))|^2]~dt<\infty,
\]
and
\[
\int_0^t[\partial_tf(a^ks,x_0+X(s))ka^{k-1}b+\partial_xf(a^ks,x_0+X(s))(y+W)]^2~ds<\infty.
\]
\end{remark}

\begin{example} 
We consider the quasi-linear equation em $L^2(\mathbb{R})$ given by \begin{eqnarray}
du(t,x)&=&[\alpha \Delta u(t,x)+\beta u(t,x)+\overrightarrow{\gamma} \cdot\nabla u(t,x)]~dt+dB_t^H   \label{eq2}\\
u(0,x)&=&v(x)\nonumber
\end{eqnarray}
for $\alpha, \beta$ non-negative constants, $\overrightarrow{\gamma}\in \mathbb{R}^n$ and $v(x)\in \mathcal{H}^2(\mathbb{R})$. 

We first observe that $\mathcal{H}^2(\mathbb{R})\subset L^2(\mathbb{R})$ is densely contained with 
\[
f(t,u)=[\alpha \Delta u +\beta u+\overrightarrow{\gamma} \cdot \nabla u]
\]
and
\[
\partial_uf(t,u)v=[\alpha \Delta v+\beta v+\overrightarrow{\gamma} \cdot \nabla v]
\]
Since the Laplace operator is linear, we get
\begin{eqnarray*}
||\Delta u||_{L^2}&\leq& C(1+||u||_{\mathcal{H}^2})\\
||u(t,x))||_{L^2}&\leq& C(1+||u||_{\mathcal{H}^2})\\
||(|\nabla u(t,x)|)||_{L^2}&\leq& C(1+||u||_{\mathcal{H}^2}),
\end{eqnarray*}for a positive constant $C$.
It follows that 
\[
|f(t,u)|+|\partial_tf(t,u)|\leq C(1+|u|)
\]
and
\[
|\partial_uf(t,u)|^2\leq C(1+|u|^2+|v|^2).
\]
By Theorem \ref{teorema} there is a weak solution for the equation (\ref{eq2}).
\end{example}

\end{document}